\documentclass[10pt,leqno]{amsart}
\usepackage{amsfonts}
\usepackage{amsmath}
\usepackage{amsthm}
\usepackage{amssymb}

\theoremstyle{definition}
\newtheorem{definition}{Definition}[section]

\theoremstyle{plain}

\newtheorem{pro}[definition]{Proposition}

\begin{document}
\def\iff{if and only if }
\def\pi{positive implicative }

\setcounter{page}{1} 
\vspace{16mm}

\begin{center}
{\normalsize \textbf{Optimal Pair of Two Linear Varieties}} \\[%
12mm]
\textsc{Armando Gon\c calves$^{1}$, M. A. Facas Vicente$^{\ast }$$^{1}$ $^{2}$ and Jos\'{e} Vit\'{o}ria$^{1}$}\\[8mm]

\begin{minipage}{123mm}
{\small {\sc Abstract.}
   The optimal pair of two linear varieties is considered as a best approximation problem, namely the distance between a point and the difference set of two linear varieties. The Gram determinant allows to get the optimal pair in closed form.}
\end{minipage}
\end{center}

\renewcommand{\thefootnote}{} \footnotetext{$^{\ast }$\thinspace
Corresponding author.} \footnotetext{%
}
 \footnotetext{$^{1}$\thinspace
Department of Mathematics, University of Coimbra, Apartado 3008, EC\ Santa
Cruz, 3001-501 Coimbra, Portugal. E-mails: \textit{adsg@mat.uc.pt} (Armando Gon\c calves), \textit{vicente@mat.uc.pt} (M. A. Facas Vicente) and \textit{jvitoria@mat.uc.pt} (Jos\'{e} Vit\'{o}ria).}
\footnotetext{$^{2}$\thinspace Supported by Instituto de Engenharia de
Sistemas e Computadores---Coimbra, Rua Antero de Quental, 199, 3000-033
Coimbra, Portugal.}

\bigskip

\section{Introduction}


\bigskip

In this paper we deal with the problem of finding the the optimal points of two linear varieties in a finite dimensional real linear space.

The distance between two linear varieties has been dealt with in several papers: \cite{Caseiro}, \cite{Dax} and \cite{LAA1}. In \cite{LAA1}, only the distance is considered and the points that realize the distance are not exhibited. In \cite{Caseiro} the best approximation points are found.

In \cite{Caseiro} projecting equations were called into play; in \cite{Dax} the difference set of two closed convex sets in $\mathrm{I\kern-.17emR}^{m}$ was considered.

In this paper, we formulate the problem as suggested by \cite[page 196]{Dax} and we use Gram theory \cite[page 74]{Luenberger}, \cite[page 65]{Deutsch} to solve it.

We use results on existence and uniqueness of optimal points by considering a least norm problem of the difference set of two closed convex sets in $\mathrm{I\kern-.17emR}^{m}$ \cite{Dax}. For the concepts, results and  motivation on the study of the distance between convex sets see \cite{Dax}.

We endow the space $\mathrm{I\kern-.17emR}^{m}$ with the usual inner
product $\bullet$:%
\begin{equation*}
\overrightarrow{p}\bullet \overrightarrow{q}:=p_{1}q_{1}+p_{2}q_{2}+\cdots
p_{m}q_{m}\, ,
\end{equation*}%
where $\overrightarrow{p}=\left[
\begin{array}{cccc}
p_{1} & p_{2} & \cdots & p_{m}%
\end{array}%
\right] ^{T}$\ and $\overrightarrow{q}=\left[
\begin{array}{cccc}
q_{1} & q_{2} & \cdots & q_{m}%
\end{array}%
\right] ^{T}$ and with the Euclidean norm%
\begin{equation*}
\left\Vert \overrightarrow{p}\right\Vert =\sqrt{\overrightarrow{p}\bullet
\overrightarrow{p}}.
\end{equation*}

\bigskip

\bigskip

\section{The Result}


\bigskip

We are looking for the optimal pair of two linear varieties $V_{\vec{b}}$ and $V_{\vec{c}}$ in $\mathrm{I\kern-.17emR}^{m}$ defined by $$V_{\vec{b}}:=\{\vec{b}+B \vec{u}:\,\vec{u} \in  \mathrm{I\kern-.17emR}^{l_1}\} \,\,\mbox{and}\,\,V_{\vec{c}}:=\{\vec{c}-C \vec{v}:\,\vec{v} \in  \mathrm{I\kern-.17emR}^{l_2}\},$$
with $\vec{b}$ and $\vec{c}$ given vectors in $\mathrm{I\kern-.17emR}^{m}$.

Following \cite[page 196]{Dax} the distance $d(V_{\vec{b}},V_{\vec{c}})$ between the two linear varieties $V_{\vec{b}}$ and $V_{\vec{c}}$ is obtained through the minimization of $\|A\vec{x}-\vec{d}\|$ where: $$\vec{d}=\vec{c}-\vec{b};\,\, A=[B\,\,C]\, \mbox{is a real}\,\, m\times(l_1 + l_2 )\,\, \mbox{matrix; and}\,\, \vec{x}= [\vec{u}^T \, \vec{v}^T]^T \,\in \mathrm{I\kern-.17emR}^{l_1 + l_2 }.$$

So the distance $d(V_{\vec{b}},V_{\vec{c}})$ between the varieties $V_{\vec{b}}$ and $V_{\vec{c}}$, may be studied using the shortest distance between a point $\vec{d}$ and the subspace $\mbox{Range} (A)$, the column space of matrix $A$.

Besides getting the distance $d(V_{\vec{b}},V_{\vec{c}})$, we also find the points $\vec{b}^* \in V_{\vec{b}}$ and $\vec{c}^* \in V_{\vec{c}}$ such that $d(V_{\vec{b}},V_{\vec{c}}) = \| \vec{b}^* - \vec{c}^* \|$, that is to say $\vec{b}^*$ and $\vec{c}^*$ are the best approximation points.

In this new setting, and using the Euclidean norm, Gram theory \cite{Luenberger} can play an important role.

Some definitions are needed.

\begin{definition}
The optimal pair of two linear varieties $\mathcal{A}$ and $\mathcal{B}$ is the pair $(\vec{a}^* , \vec{b}^* ) \in \mathcal{A} \times \mathcal{B}$ satisfying $d(\mathcal{A} , \mathcal{B} ) = \| \vec{a}^* - \vec{b}^* \|$, where $d(\mathcal{A} , \mathcal{B} )$, the distance between $\mathcal{A}$ and $\mathcal{B}$, is defined as
$$d(\mathcal{A} , \mathcal{B} ) = \inf\{\|\vec{a} - \vec{b}\| : \vec{a} \in \mathcal{A}, \vec{b} \in \mathcal{B}\}.$$
\end{definition}

\begin{definition}
Let $y_1 , y_2 , \ldots, y_n$ be elements of $\mathrm{I\kern-.17emR}^{m}$. The $n\times n$ matrix
$$G(y_1 , y_2 , \ldots, y_n) = \left[
\begin{array}{cccc}
y_{1}\bullet y_1 & y_1 \bullet y_{2} & \cdots & y_1 \bullet y_{n}\\
y_{2} \bullet y_1 & y_2 \bullet y_{2} & \cdots & y_2 \bullet y_{n}\\
\vdots & \vdots & \vdots & \vdots\\
y_{n} \bullet y_1 & y_n \bullet y_{2} & \cdots & y_n \bullet y_{n}
\end{array}%
\right]$$
is called the Gram matrix of $y_1 , y_2 , \ldots, y_n$. The determinant $g(y_1 , y_2 , \ldots, y_n )$ of the Gram matrix is known as the Gram determinant.
\end{definition}

It is known, see for example \cite{Laurent}, \cite{Luenberger}, that $g(y_1 , y_2 , \ldots, y_n ) \geq 0$ and $g(y_1 , y_2 , \ldots, y_n )=0$ if and only if $y_1 , y_2 , \ldots, y_n $ are linearly dependent.

\begin{pro}
Let be given the linear varieties
$$V_{\vec{b}}:=\{\vec{b}+B \vec{u}:\,\vec{u} \in  \mathrm{I\kern-.17emR}^{l_1}\}\, , \,V_{\vec{c}}:=\{\vec{c}-C \vec{v}:\,\vec{v} \in  \mathrm{I\kern-.17emR}^{l_2}\},$$
where $\vec{b}$ and $\vec{c}$ are any vectors in $\mathrm{I\kern-.17emR}^{m}$ and $B\in \mathrm{I\kern-.17emR}^{m\times l_1}\, ,\, C\in \mathrm{I\kern-.17emR}^{m\times l_2}$ are fixed matrices.

Let consider $\vec{d}= \vec{c}-\vec{b}$ and assume that $A=[B\,\,\,C]=[\vec{a_1}\,\vec{a_2}\, \ldots\,\vec{a_n}]\in \mathrm{I\kern-.17emR}^{m\times n}$, $n=l_1 + l_2$, is a full column rank matrix. Then:
\begin{enumerate}
\item[(A)] The optimal pair $(\vec{b}^* , \vec{c}^* )$ is obtained as
$$\vec{b}^* = \vec{b}+B\vec{u}^* \, , \, \vec{c}^* = \vec{c} - C\vec{v}^* ,$$
with

\begin{eqnarray}
\left[ \begin{array}{c}
\vec{u}^*\\
\vec{v}^*
\end{array}%
\right] & = & - \frac{1}{g(\vec{a_1},\vec{a_2},\ldots,\vec{a_n})}\left| \begin{array}{cc}
\begin{array}{c}
G(\vec{a_1},\vec{a_2},\ldots,\vec{a_n})
\end{array} & \!\left|\begin{array}{c}
                                                        \vec{d} \bullet \vec{a_1}\\
                                                        \vdots\\
                                                        \vec{d} \bullet \vec{a_n}
                                           \end{array}\right.\\
                                           \hline & {}\\
                                           \begin{array}{cccc}
                                           \vec{a_1} & \vec{a_2} & \ldots & \vec{a_n}
                                           \end{array} & \!\left|\begin{array}{c}
                                                        \,\,\,\,\,\,\vec{0}\,\,\,\,\,\,
                                           \end{array}\right.\\
\end{array}
\right|,
\end{eqnarray} \label{eq1}
where the (formal) determinant is to be expanded by the last row to yield a linear combination of the columns $\vec{a_1}, \vec{a_2}, \ldots, \vec{a_n}$ of the matrix $A$.
\item[(B)] The distance between $V_{\vec{b}}$ and $V_{\vec{c}}$ is given by
$$d(V_{\vec{b}}, V_{\vec{c}}) = \|\vec{b}^* - \vec{c}^{\,*}\|$$
and also

\begin{eqnarray}
d^2 (V_{\vec{b}}, V_{\vec{c}}) & = & \frac{g(\vec{d},\vec{a_1},\vec{a_2},\ldots,\vec{a_n})}{g(\vec{a_1},\vec{a_2},
\ldots,\vec{a_n})}.
\end{eqnarray} \label{eq2}

\end{enumerate}
\end{pro}

\begin{proof} From \cite[page 196]{Dax} we must minimize $\left\Vert A\overrightarrow{x}-\overrightarrow{d} \right\Vert$, where $\overrightarrow{x} = \left[ \begin{array}{c}
\vec{u}^*\\
\vec{v}^*
\end{array} \right]$ and this is equivalent to find the distance from the point $\overrightarrow{d}$ to the column-space Range(A) of A.

By hypothesis, the columns of the matrix A are linearly independent. Then from \cite[page 74]{Luenberger} we obtain (1) and from \cite[page 65]{Deutsch} we obtain (2).
\end{proof}

\end{document}